\documentclass[11pt]{article}

\usepackage{amsmath}
\usepackage{amssymb}
\usepackage{amsthm}
\usepackage{titlesec}
\usepackage{graphicx}
\usepackage{caption}
\usepackage{subcaption}

\usepackage[T1]{fontenc}
\usepackage[utf8]{inputenc}

\usepackage{comment}

\usepackage{indentfirst}
\usepackage[top=1.25in,left=1.25in,right=1.25in]{geometry}
\newlength{\numone}
\newlength{\widone}
\newlength{\numtwo}
\newlength{\widtwo}
\newcounter{countp}

\newtheorem{thm}{Theorem}
\newtheorem{thmp}[countp]{Theorem}
\newtheorem{lemma}[thm]{Lemma}
\theoremstyle{definition}
\newtheorem{defin}[thm]{Definition}
\newtheorem{alg}[thm]{Algorithm}
\newtheorem{rem}[thm]{Remark}

\numberwithin{equation}{section}

\author{\Large{Riccardo W. Maffucci\footnote{\textsc{EPFL, MA SB Batiment 8, Lausanne, Switzerland. }\texttt{riccardo.maffucci@epfl.ch}.}}}
\title{\Large{\uppercase{\bf Constructing certain families of $\mathbf{3}$-polytopal graphs}}}

\date{}

\def\calP{\mathcal{P}}
\def\calQ{\mathcal{Q}}
\def\calR{\mathcal{R}}

\begin{document}
\titleformat{\section}
  {\Large\scshape}{\thesection}{1em}{}
\titleformat{\subsection}
  {\large\scshape}{\thesubsection}{1em}{}
\maketitle


\begin{abstract}
Let $n\geq 3$ and $r_n$ be a $3$-polytopal graph such that for every $3\leq i\leq n$, $r_n$ has at least one vertex of degree $i$. We find the minimal vertex count for $r_n$. We then describe an algorithm to construct the graphs $r_n$. A dual statement may be formulated for faces of $3$-polytopes. The ideas behind the algorithm generalise readily to solve related problems.

Moreover, given a $3$-polytope $t_l$ comprising a vertex of degree $i$ for all $3\leq i\leq l$, $l$ fixed, we define an algorithm to output for $n>l$ a $3$-polytope $t_n$ comprising a vertex of degree $i$, for all $3\leq i\leq n$, and such that the initial $t_l$ is a subgraph of $t_n$. The vertex count of $t_n$ is asymptotically optimal, in the sense that it matches the aforementioned minimal vertex count up to order of magnitude, as $n$ gets large. In fact, we only lose a small quantity on the coefficient of the second highest term, and this quantity may be taken as small as we please, with the tradeoff of first constructing an accordingly large auxiliary graph.

\end{abstract}


\section{Introduction}
\subsection{The question}
Graphs that are planar and $3$-connected have the nice property of being $1$-skeletons of $3$-polytopes, as proven by Rademacher-Steinitz (see e.g. \cite[Theorem 11.6]{harary}). We call these graphs $3$-polytopal graphs, or $3$-polytopes, or polyhedra interchangeably. These special planar graphs are \textit{uniquely} embeddable in a sphere (as observed by Whitney, see e.g. \cite[Theorem 11.5]{harary}). Their regions are also called `faces', and are delimited by cycles (polygons) \cite[Proposition 4.26]{dieste}.

Our starting point is the following question. Let $n\geq 3$ and $G'$ be a $3$-polytopal graph such that for every $3\leq i\leq n$, $G'$ has at least one $i$-gonal face. What is the minimal number of faces for $G'$?

In what follows, we will work on the \textit{dual} problem. Indeed, it is well-known that $3$-polytopes have a \textit{unique} dual graph, that is also $3$-polytopal (see e.g. \cite[Chapter 11]{harary}).

\begin{defin}
	\label{def:1}
A $3$-polytope has the \textit{property $\calP_n$} if it has at least one vertex of degree $i$, for each $3\leq i\leq n$, and moreover it has minimal order (number of vertices) among $3$-polytopes satisfying this condition.
\end{defin}

The notation $H\prec G$ indicates that $H$ is a subgraph of $G$. Our first result is the following.
\setcounter{countp}{\thethm}
\begin{thm}
\label{thm:1}
Let $n\geq 3$ and $G$ be a $3$-polytopal graph with at least one vertex of degree $i$, for every $3\leq i\leq n$. Then the minimal number $p(n)$ of vertices of $G$ is
\begin{equation}
\label{eqn:p}
p(n)=\left\lceil\frac{n^2-11n+62}{4}\right\rceil, \qquad\forall n\geq 14.
\end{equation}
For $n\leq 13$, we have the values in Table \ref{tab:1}.
\begin{table}[h!]
	\centering
	$\begin{array}{|c||c|c|c|c|c|c|c|}
	\hline
	n&3\leq n\leq 7
	&8&9&10&11&12&13\\\hline
	p(n)&n+1
	&10&11&14&16&19&23\\\hline
	\end{array}$
	\caption{Values of $p(n)$ for $n\leq 13$.}
	\label{tab:1}
\end{table}
\\
Moreover, starting from $n=14$ and for every $n\geq 16$, Algorithm \ref{alg:1} constructs a $3$-polytope $r_n$ satisfying $\calP_n$ and the relations
\begin{equation}
\label{eqn:sub}
r_n\succ
\begin{cases}
r_{n-2} & \text{ if } n\equiv 0 \pmod 4,
\\r_{n-3} & \text{ if } n\equiv 1 \pmod 4,
\\r_{n-4} & \text{ if } n\equiv 2 \pmod 4,
\\r_{n-5} & \text{ if } n\equiv 3 \pmod 4
\end{cases}
\qquad n\geq 16.
\end{equation}
Graphs satisfying $\calP_n$ for $3\leq n\leq 15$ are depicted in Figures \ref{fig:small}, \ref{fig:8-13}, \ref{fig:14}, and \ref{fig:15}.
\end{thm}

Theorem \ref{thm:1} will be proven in sections \ref{sec:lb}, \ref{sec:sm}, and \ref{sec:alg1}. Passing to the duals, we can answer the original question.
\begin{thmp}
If $n\geq 3$ and $G'$ is a $3$-polytopal graph with at least one $i$-gonal face, for every $3\leq i\leq n$, then the minimal number $p(n)$ of faces for $G'$ is given by \eqref{eqn:p} and Table \ref{tab:1}.
\end{thmp}

\subsection{A related problem}
In our next result, given a $3$-polytope $H$ containing vertices of valencies $\{3,4,\dots, l\}$, $l\geq 5$, and an integer $n>l$, we aim to construct a $3$-polytope $G$ containing a copy of $H$ as subgraph, and comprising vertices of degrees $\{3,4,\dots, n\}$. We start with the following definition.
\begin{defin}
	\label{def:2}
Let $n\geq 5$ be odd. We say that a $3$-polytope satisfies \textit{property $\calQ_n$} if there is at least one vertex of degree $3$, and moreover the polytope contains among its faces the triangles
\begin{equation}
\label{eqn:faces}
F_{n;j}=\{v_{n;j;1}, v_{n;j;2}, v_{n;j;3}\}, \quad 1\leq j\leq(n-3)/2
\end{equation}
where
\begin{equation}
\label{eqn:odd}
\{\deg(v_{n;j;1}) : 1\leq j\leq(n-3)/2\} \cup \{\deg(v_{n;j;2}) : 1\leq j\leq(n-3)/2\}=\{4,5,\dots,n\}
\end{equation}
and
\[v_{n;j;3}\neq v_{n;i;1},v_{n;i;2},  \quad\forall j,i.\]
\end{defin}

Note that $\calQ_n$ together with minimality w.r.t. order is stronger than property $\calP_n$ of Definition \ref{def:1}.

\begin{thm}
	\label{thm:2}
Let $l\geq 5$ be odd, and $t_l$ be a $3$-polytope satisfying $\calQ_l$. Fix an integer $m\geq 14$, $m\equiv 2 \pmod 4$, and let $n:=l+mk$, where $k$ is a non-negative integer. Then there exists a sequence of $3$-polytopes
\[\{t_n\}_{n=l+mk, \ k\geq 1},\]
where each $t_n$ satisfies $\calQ_{n}$, such that along the sequence, for all $\epsilon>0$, one has
\begin{equation}
\label{eqn:ordt}
|V(t_n)|\leq\frac{n^2}{4}-\frac{11n}{4}+\left(\frac{5}{2m}+\epsilon\right)n
\end{equation}
as $n\to\infty$. Moreover, for all $k\geq 1$, it holds that
\begin{equation}
\label{eqn:subt}
t_{l+m(k-1)}\prec t_{l+mk}.
\end{equation}
For chosen $t_l$ and $N\geq 1$, Algorithm \ref{alg:2} constructs $t_{l+mk}$ for $1\leq k\leq N$.
\end{thm}

Theorem \ref{thm:2} will be proven in section \ref{sec:ind}.

\begin{rem}
The order \eqref{eqn:ordt} of the sequence of polyhedra in Theorem \ref{thm:2} is asymptotically optimal, in the sense that the leading term is $n^2/4$ as in \eqref{eqn:p}. The coefficient of the linear term is only slightly larger than the $-11n/4$ of \eqref{eqn:p}. This difference can be taken as small as we please if $m$ is chosen to be large, with the tradeoff of first constructing an accordingly large auxiliary graph $s_{m+3}$, as detailed in sections \ref{sec:ap} and \ref{sec:ind}.
\end{rem}

\begin{rem}
In Definition \ref{def:2}, we could have supposed instead $n$ even. Then accordingly one would have taken $j\leq (n-2)/2$ in \eqref{eqn:faces}, and the set $\{3,4,\dots, n\}$ on the RHS of \eqref{eqn:odd}. This would have produced similar setup and ideas.
\end{rem}

\subsection{Related literature, notation, and plan of the paper}

\paragraph{Related literature.}
Necessary conditions for the degree sequence of a planar graph were given in \cite{bowen1,chvata}. On the other hand, Eberhard \cite{eberha} proved that any degree sequence where $q=3p-6$ ($p,q$ being vertex and edge counts respectively) may be made planar by inserting a sufficiently large number of $6$'s. There have been numerous generalisations and extensions since, see e.g. \cite{barnet,badagr,fisher,grunba,jendr2,jendr1}. In \cite{hhrt66}, the authors determine the sequences for regular, planar graphs. This was extended in \cite{schhak} to sequences with highest and lowest valencies differing by one or two.


\paragraph{Notation.}
All graphs that appear contain no loops and multiple edges. The vertex and edge sets of a graph $G$ are denoted by $V(G)$ and $E(G)$ respectively. The order and size of $G$ are the numbers $|V(G)|$ and $|E(G)|$. The degree or valency $\deg_G(v)$ of a vertex $v$ counts the number of vertices adjacent to $v$ in $G$. We use the shorthand $\deg(v)$ when $G$ is clear. The degree sequence of $G$ is the set of all vertex valencies.
\\
We write $G\cong H$ when $G,H$ are isomorphic graphs, and $H\prec G$ when $H$ is (isomorphic to) a subgraph of $G$.
\\
A graph of order $k+1$ or more is said to be $k$-connected if removing any set of $k-1$ or fewer vertices produces a connected graph.
\\
Regions of a $2$-connected planar graph are cycles of length $i$ ($i$-gons) \cite[Proposition 4.26]{dieste}. For these graphs, the terms `region' and `face' are interchangeable. The $i$-gonal faces will be denoted by their sets of $i$ vertices. If $\{a,b,c\}$ is a triangle, we call \textit{splitting} the operation of adding a vertex $d$ and edges $da, db, dc$.


\paragraph{Plan of the paper.}
Theorem \ref{thm:1} is proven in sections \ref{sec:lb} (lower bound \eqref{eqn:p}), \ref{sec:sm} (cases $n\leq 13$) and \ref{sec:alg1} (Algorithm \ref{alg:1} for $n\geq 14$). Section \ref{sec:ap} is about an application of a similar flavour, that we can tackle via a minor modification of Algorithm \ref{alg:1}. The theory of section \ref{sec:ap} will also be useful in section \ref{sec:ind} to prove Theorem \ref{thm:2}.
\\
Appendix \ref{appa} presents another way to think about Algorithm \ref{alg:1}. Appendix \ref{appb} contains figures of graphs mentioned in sections \ref{sec:sm} and \ref{sec:alg1}.

\subsection{Acknowledgements}
The author was supported by Swiss National Science Foundation project 200021\_184927.


\section{The lower bound}
\label{sec:lb}
In this section we prove \eqref{eqn:p}.
\begin{lemma}
	\label{le:lb}
Let $n\geq 3$ and $G(n)$ be a $3$-polytopal graph with at least one vertex of degree $i$, for every $3\leq i\leq n$. Then its order is at least
\begin{equation}
\label{eqn:6}
p(n)\geq\left\lceil\frac{n^2-5n+30}{6}\right\rceil.
\end{equation}
Moreover, as soon as $n\geq 8$, we also have
\begin{equation}
\label{eqn:4}
p(n)\geq\left\lceil\frac{n^2-11n+62}{4}\right\rceil.
\end{equation}
\end{lemma}
\begin{proof}
Let $p=p(n)$, $q=q(n)$ denote order and size of $G(n)$, and $d_j:=\deg(v_j)$. On one hand, via the handshaking lemma,
\[2q=\sum_{i=3}^{n}i+\sum_{j=n-1}^{p}d_j\geq\frac{n(n+1)}{2}-3+3(p-n+2)=\frac{(n-2)(n-3)}{2}+3p\]
where we used $3$-connectivity. On the other hand, by planarity, $q\leq 3p-6$, so that altogether
\[p\geq\frac{n^2-5n+30}{6}\]
hence \eqref{eqn:6}.
\\
By \cite{bowen1,chvata}, for any $3\leq k\leq (p+4)/3$, it holds that
\begin{equation}
\label{eqn:bc}
\sum_{i=1}^{k}d_i\leq 2p+6k-16.
\end{equation}
To optimise this lower bound for $p$, the left hand side should contain as many numbers exceeding $5$ as possible. We thus wish to take $k=n-5$, and we may do this as long as $3\leq n-5\leq (p+4)/3$, i.e.,
\[n\geq 8 \quad\text{ and }\quad p\geq 3n-19.\]
By \eqref{eqn:6}, these conditions certainly hold for all $n\geq 8$. In this case, equation \eqref{eqn:bc} with $k=n-5$ reads
\[\frac{n(n+1)}{2}-15\leq 2p+6(n-5)-16,\]
and rearranging this inequality we obtain \eqref{eqn:4}.
\end{proof}

\section{Proof of Theorem \ref{thm:1} for $n\leq 13$}
\label{sec:sm}
The polyhedra up to $8$ faces were tabulated in \cite{brdu73} and \cite{fede75}. For $4\leq n+1\leq 8$, we are looking for polyhedra with at least one $i$-gonal face for each $3\leq i\leq n$. We consult \cite[Table I]{fede75}, searching where the `Faces' column has the maximal $n-2$ non-zero entries. It is straightforward to find the ten relevant cases (numbered 1, 2, 3, 4, 5, 13, 14, 15, 46, and 47 in \cite{fede75}). Passing to the dual graphs, we obtain the ten $3$-polytopes sketched in Figure \ref{fig:small}. In particular, for $3\leq n\leq 7$, we have $p(n)=n+1$ (cf. Table \ref{tab:1}).

Next, we wish to find examples of polyhedra satisfying $\calP_n$ for $8\leq n\leq 13$. We observe that each graph in Figure \ref{fig:small}, save for the tetrahedron and square pyramid, may be obtained from a previous one via a splitting operation. Our strategy is then to apply repeated splitting on the faces of $r_{7.1}$ from Figure \ref{fig:small}, to obtain a new graph $r_n$. For $8\leq n\leq 13$, the aim is to obtain a subset of vertices of valencies $4,5,\dots,n$. In the effort to minimise the resulting graph's order, we split triangles of $r_{7.1}$ containing the maximal possible number of vertices of degree $4$ or more, ideally all three of them. 
We thereby construct the graphs of Figure \ref{fig:8-13}. Their orders match the largest of the lower bounds \eqref{eqn:6} and \eqref{eqn:4} proven in section \ref{sec:lb}. We thus complete Table \ref{tab:1}.

In the next section, we will combine the above with other ideas to write Algorithm \ref{alg:1}, proving the cases $n\geq 14$ of Theorem \ref{thm:1}.

\section{Proof of Theorem \ref{thm:1} for $n\geq 14$}
\label{sec:alg1}
\subsection{Setup}
Let $r_{14}$ be the graph sketched in Figure \ref{fig:14}.
\begin{figure}[h!]
	\centering
	\includegraphics[width=8cm,clip=false]{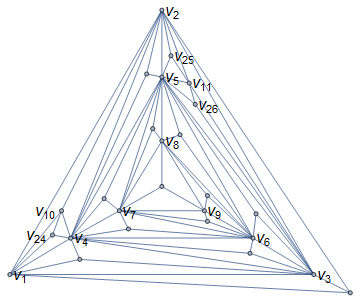}
	\caption{The $3$-polytope $r_{14}$, satisfying $\calP_{14}$.} 
	\label{fig:14}
\end{figure}
It is straightforward to check that $r_{14}$ is a polyhedron, and that the respective valencies of $v_j$, $1\leq j\leq 11$, are
\[7,9,11,13,14,12,10,8,6,4,5.\]
The order of this graph is $26$, matching the lower bound \eqref{eqn:4} in the case $n=14$, and there are vertices of degree $3$ as well. Theorem \ref{thm:1} is hence proven in this case. In the following we recursively construct the $r_n$, $n\geq 16$, of Theorem \ref{thm:1}. As for $r_{15}$ (Figure \ref{fig:15}), we obtain it from $r_{14}$ via one edge deletion and again applying the ideas of section \ref{sec:sm}.

We need a preliminary definition, the operation of $\mathit{h}$-\textit{splitting} a triangle about a vertex, for some $h\geq 1$ (see Figure \ref{fig:spl}). To $h$-split $\{a,b,c\}$ about $c$, we begin by splitting it, introducing a new vertex $c_1$. Then we split $\{a,b,c_1\}$ inserting $c_2$, and so on, until we have added the vertex $c_h$. For instance referring to Figure \ref{fig:small}, given the tetrahedron $s_3$, $1$-splitting any face yields $s_{4.2}$, while $2$-splitting any face produces $s_{5.2}$.
\begin{figure}[h!]
	\centering
	\includegraphics[width=3cm,clip=false]{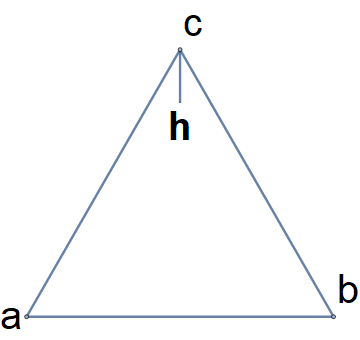}
	\caption{Notation for $h$-splitting a triangle $\{a,b,c\}$ about the vertex $c$.}
	\label{fig:spl}
\end{figure}

\subsection{The case $n\equiv 2 \pmod 4$}
We now describe the algorithm producing the $r_n$ of Theorem \ref{thm:1}, starting from the case $n\equiv 2 \pmod 4$. The remaining cases are covered in section \ref{sec:rem}.
\begin{alg}[Part I]
\label{alg:1}\

\textbf{Input.} An integer $N\geq 16$.

\textbf{Output.} A set of graphs $\{r_{n} : 16\leq n\leq N\}$, where each $r_n$ has property $\calP_n$.

\textbf{Description.}
For all $k\geq 1$, we define the graph $\text{pc}_{k}$ (`$k$-th piece') of Figure \ref{fig:piece}. 
The half-lines and numbers in bold represent $h$-splitting: for instance, face $\{a_k,b_k,e_k\}$ is split $h=4k$ times about the vertex $e_k$. Letting $\text{pc}_{0}:=r_{14}$, we label $u_0,v_0,w_0$ its vertices of degrees $10,8,6$ respectively ($v_7,v_8,v_9$ is Figure \ref{fig:14}). The vertex (of degree $3$) adjacent to these three will be denoted by $x_0$. Note that also in each $\text{pc}_{k}$, $k\geq 1$, there are vertices $u_k,v_k,w_k$ of degrees $10,8,6$, and $x_k$ of degree $3$ adjacent to them. We define
\begin{equation}
\label{eqn:cl1}
r_n:=r_{14}\cup\bigcup_{k=1}^{(n-14)/4}\text{pc}_{k}, \qquad n\geq 14, \ n\equiv 2 \pmod 4,
\end{equation}
identifying in each union operation the vertices $u_{k-1},v_{k-1},w_{k-1},x_{k-1}$ with $a_k,b_k,c_k,d_k$ respectively.
\begin{figure}[h!]
	\centering
	\includegraphics[width=12cm,clip=false]{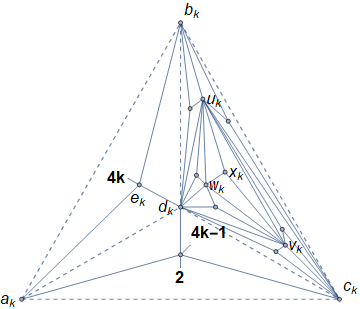}
	\caption{The graph $\text{pc}_{k}$. Dashed lines are not edges of $\text{pc}_{k}$. Half-lines and numbers in bold represent $h$-splitting.}
	\label{fig:piece}
\end{figure}
\end{alg}

\begin{proof}[Proof of Theorem \ref{thm:1} for $n\geq 14$, $n\equiv 2 \pmod 4$]
We claim that $r_n$ \eqref{eqn:cl1} has order \eqref{eqn:p} and satisfies property $\calP_n$. We argue by induction. The base case $n=14$ has already been checked. The union \eqref{eqn:cl1} is still a $3$-polytope by construction. By the inductive hypothesis, the graph
\begin{equation*}
r_{n-4}=r_{14}\cup\bigcup_{k=1}^{(n-18)/4}\text{pc}_{k}
\end{equation*}
satisfies $\calP_{n-4}$, and has order
\[p(n-4)=\frac{n^2-19n+132}{4}.\]

Turning to $r_n=r_{n-4}\cup\text{pc}_{(n-14)/4}$, we record that for $\text{pc}_k\prec r_n$, $k\geq 1$, the vertices $a_k,b_k,c_k,d_k$ have respective valencies
\begin{align*}
&\deg_{r_n}(a_k)=10+1+4k+1+2=4k+14,  \\&\deg_{r_n}(b_k)=8+1+4k+3=4k+12,  \\&\deg_{r_n}(c_k)=6+1+(4k-1)+2+5=4k+11,  \\&\deg_{r_n}(d_k)=3+1+1+(4k-1)+7=4k+13.
\end{align*}
In particular,
\begin{align*}
&\deg_{r_n}(a_{(n-14)/4})=n,   &\deg_{r_n}(b_{(n-14)/4})=n-2,  \\ &\deg_{r_n}(c_{(n-14)/4})=n-3,   &\deg_{r_n}(d_{(n-14)/4})=n-1.
\end{align*}
The degree of the remaining vertices of $r_{n-4}$ has not changed in the union. Moreover, $\deg_{r_n}(u_{(n-14)/4})=10$, $\deg_{r_n}(v_{(n-14)/4})=8$, and $\deg_{r_n}(w_{(n-14)/4})=6$. As for order, $\text{pc}_{(n-14)/4}$ introduces $12$ new vertices plus those given by the $h$-splittings, namely $4(n-14)/4+2+4(n-14)/4-1=2n-27$. It follows that
\[|V(r_n)|=p(n-4)+12+2n-27=\frac{n^2-11n+62}{4}.\]
Therefore, $r_n$ does indeed have property $\calP_n$. Moreover, it is clear from the construction that $r_{n-4}\prec r_n$.
\end{proof}

\subsection{The cases $n\equiv 1,0,3 \pmod 4$}
\label{sec:rem}
\addtocounter{thm}{-1}
\begin{alg}[Part II]
Continuing Algorithm \ref{alg:1}, for $n\geq 16$, $n\equiv 1$ (resp. $\equiv 0$) $\pmod 4$, we start by constructing $r_{n-3}$ (resp. $r_{n-2}$) as above. We then define
\begin{equation*}
r_n:=r_{n-3}\cup\text{end}_{1;n}
\end{equation*}
(resp. $r_n:=r_{n-2}\cup\text{end}_{0;n}$) with $\text{end}_{1;n}$ (resp. $\text{end}_{0;n}$) given by Figure \ref{fig:end1} (resp. \ref{fig:end0}).
\\
For $n\equiv 1$, in the union $r_n:=r_{n-3}\cup\text{end}_{1;n}$, vertices $u_{(n-17)/4},v_{(n-17)/4},w_{(n-17)/4},x_{(n-17)/4}$ are identified with $a,b,c,d$ respectively. For $n\equiv 0$, in the union $r_n:=r_{n-2}\cup\text{end}_{0;n}$, vertices $u_{(n-16)/4},v_{(n-16)/4},w_{(n-16)/4},x_{(n-16)/4}$ are identified with $\bar{a},\bar{b},\bar{c},\bar{d}$ respectively.
\\Finally, if $n\equiv 3 \pmod 4$, we take
\[r_n:=r_{n-2}\cup\text{end}_{0;n}=r_{n-5}\cup\text{end}_{1;{n-2}}\cup\text{end}_{0;n},\]
identifying $d,v,w,x$ of $\text{end}_{1;{n-2}}$ with $\bar{a},\bar{b},\bar{c},\bar{d}$ of $\text{end}_{0;n}$ respectively.

\begin{figure}[h!]
\begin{subfigure}{0.49\textwidth}
	\centering
	\includegraphics[width=5cm,clip=false]{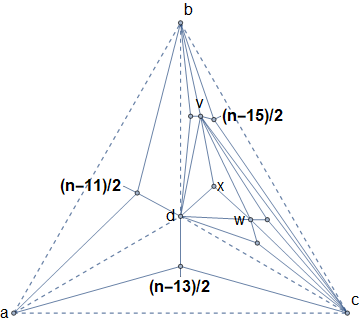}
	\caption{The graph $\text{end}_{1;n}$.}
	\label{fig:end1}
\end{subfigure}
\begin{subfigure}{0.49\textwidth}
	\centering
	\includegraphics[width=5cm,clip=false]{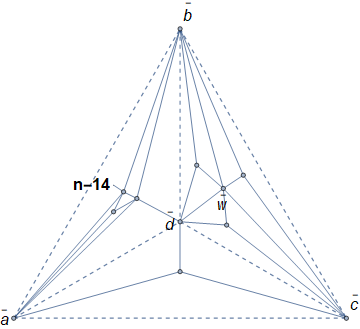}
	\caption{The graph $\text{end}_{0;n}$.}
	\label{fig:end0}
\end{subfigure}
\caption{The graphs $\text{end}_{1;n}$ and $\text{end}_{0;n}$. Half-lines and numbers in bold represent $h$-splitting.}
\end{figure}
\end{alg}

\begin{proof}[Proof of Theorem \ref{thm:1} for $n\geq 16$, $n\equiv 1,0,3 \pmod 4$]
If $n\equiv 1$, in $r_n$ one has $\deg(a)=n$, $\deg(b)=n-1$, and $\deg(c)=n-2$. The remaining valencies $\geq 4$ of $r_{n-3}$ do not change in the union $r_n$. Moreover, $\deg(d)=10$, $\deg(v)=8$, and $\deg(w)=6$. Lastly, by Theorem \ref{thm:1} in the already proven case $n\equiv 2$,
\[|V(r_n)|=p(n-3)+9+(n-11)/2+(n-13)/2+(n-15)/2=\frac{n^2-11n+62}{4}\]
as required.

Similarly, if $n\equiv 0$, in $r_n$ it holds that $\deg(\bar{a})=n$, $\deg(\bar{b})=n-1$, $\deg(\bar{c})=10$, $\deg(\bar{d})=8$, and $\deg(\bar{w})=6$. Further,
\[|V(r_n)|=p(n-2)+8+(n-14)=\left\lceil\frac{n^2-11n+62}{4}\right\rceil.\]

If $n\equiv 3$, in $r_n$ we have $\deg(a)=n-2$, $\deg(b)=n-3$, $\deg(c)=n-4$, $\deg(\bar{a})=n$, $\deg(\bar{b})=n-1$, $\deg(\bar{c})=10$, $\deg(\bar{d})=8$, and $\deg(\bar{w})=6$. Via the already proved case $n\equiv 1$,
\[|V(r_n)|=p(n-2)+8+(n-14)=\left\lceil\frac{n^2-11n+62}{4}\right\rceil.\]
The relations \eqref{eqn:sub} are clear by construction. This concludes the proof of Theorem \ref{thm:1}.
\end{proof}

\begin{rem}
The time to implement Algorithm \ref{alg:1} is quadratic in $n$, and this is optimal in the sense that the order of $r_n$ is itself quadratic \eqref{eqn:p}. Moreover, constructing $r_N$ takes no more time than obtaining all of
\[r_{14},r_{18},\dots,r_{N-(N\hspace{-0.25cm}\mod 4+2)},r_N\]
(cf. \eqref{eqn:sub}).
\end{rem}

\section{Another application}
\label{sec:ap}
The ideas of the preceding sections have solved the problem of finding for all $n$ a graph satisfying $\calP_n$ of Definition \ref{def:1}. The following lemma constitutes an application of the same ideas, and it illustrates how a minor modification of Algorithm \ref{alg:1} allows to answer similar questions. Moreover, the result of the lemma will be needed in section \ref{sec:ind}.

\begin{lemma}
	\label{le:ap}
	Let $n\geq 17$, $n\equiv 1\pmod 4$, and $H$ be an order $p$ polyhedral graph with at least one vertex of degree $i$, $3\leq i\leq n$, and at least three vertices of degree $n-1$. Then its minimal order is
	\begin{equation}
	\label{eqn:4'}
	\frac{n^2-7n+34}{4}.
	\end{equation}
\end{lemma}
\begin{proof}
Let us show the lower bound first. Similarly to Lemma \ref{le:lb}, we begin by imposing
\[6p-12\geq n+3(n-1)+\frac{(n-2)(n-1)}{2}-3+\sum_{j=n}^{p}d_j\geq\frac{n^2-n-10}{2}+3p,\]
leading to
\[p\geq\frac{n^2-n+14}{6}.\]
Thereby, for $n\geq 17$, we certainly have $3\leq n-4\leq (p+4)/3$. Applying \eqref{eqn:bc} with $k=n-4$ yields the inequality
\[n+3(n-1)+\frac{(n-2)(n-1)}{2}-21\leq 2p+6(n-4)-16,\]
and rearranging this inequality and imposing $n\equiv 1 \pmod 4$ we obtain \eqref{eqn:4'} as a lower bound.

To show the upper bound, we will actually construct a graph $s_n$ satisfying the assumptions of the present lemma, of order \eqref{eqn:4'}. We set
\begin{equation}
\label{eqn:t}
s_n:=r_{n}\cup\text{end}'_n \qquad n\geq 17, \ n\equiv 1\pmod 4
\end{equation}
where $r_n$ was constructed in Algorithm \ref{alg:1} and $\text{end}'_n$ is depicted in Figure \ref{fig:endp}.

\begin{figure}[h!]
	\centering
	\includegraphics[width=5cm,clip=false]{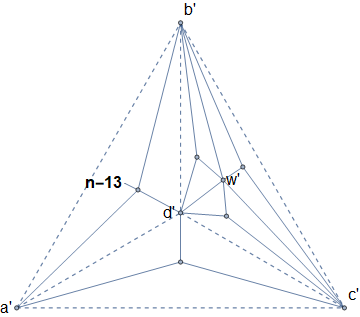}
	\caption{The graph $\text{end}'_{n}$. Dashed lines are not edges of $\text{end}'_{n}$.}
	\label{fig:endp}
\end{figure}

Since $n\equiv 1\pmod 4$,
\begin{equation*}
s_n=r_{n-3}\cup\text{end}_{1;n}\cup\text{end}'_n.
\end{equation*}
While performing the union, we identify vertices $d,v,w,x$ of $\text{end}_{1;n}$ (Figure \ref{fig:end1}) with $a',b',c',d'$ of $\text{end}'_n$ in this order. Then $s_n$ is clearly still a polyhedron. Further,
$\deg_{s_n}(a')=10+2+n-13=n-1$,  $\deg_{s_n}(b')=8+4+n-13=n-1$, $\deg_{s_n}(c')=6+4=10$, $\deg_{s_n}(d')=3+5=8$, and $\deg_{s_n}(w')=6$. Since $r_n$ has the property $\calP_n$, then $s_n$ has vertices of each degree $i$, $3\leq i\leq n$. The union with $\text{end}'_n$ has inserted two more vertices of valency $n-1$. Lastly,
\begin{equation}
|V(s_n)|=|V(r_n)|+6+(n-13)=\frac{n^2-11n+62+4n-28}{4}=\frac{n^2-7n+34}{4}.
\end{equation}
The proof of Lemma \ref{le:ap} is complete.
\end{proof}

\begin{rem}
\label{rem:ap}
We record the following property of the graph $s_n$. It plainly has degree $3$ vertices, and contains among its faces the triangles
\begin{equation}
\label{eqn:faces2}
F_{n;j}=\{v_{n;j;1}, v_{n;j;2}, v_{n;j;3}\}, \quad 1\leq j\leq(n-1)/2
\end{equation}
where $\deg(v_{n;1;1})=\deg(v_{n;1;2})=n-1$,
\[\{\deg(v_{n;j;1}) : 2\leq j\leq(n-1)/2\} \cup \{\deg(v_{n;j;2}) : 2\leq j\leq(n-1)/2\}=\{4,5,\dots,n\}\]
and
\[v_{n;j;3}\neq v_{n;i;1},v_{n;i;2},  \quad\forall j,i.\]
This property, stronger than $\calQ_n$ of Definition \ref{def:2}, shall be denoted by $\calR_n$. Indeed for $s_n$, $\calR_n$ is easily observed by construction. For instance, we may take the pairs
\begin{align*}
&(7,4), (13,12), (9,5), (14,11), \\&(4k+14,4k+12), (4k+11,4k+13) \quad \text{ for } 1\leq k\leq(n-17)/4, \\&(n,n-1), (n-2,10), (n-1,n-1), (8,6)
\end{align*}
where the notation $(a,b)$ means that if $u,v$ denote two vertices of one of the triangles $F_{n;j}$, then $\deg(u)=a$ and $\deg(v)=b$. 
\end{rem}

\section{Proof of Theorem \ref{thm:2}}
\label{sec:ind}
\subsection{Premise}
\label{sec:pre}
We first introduce the main ideas of the proof, via the following lemma.
\begin{lemma}
	\label{le:2}
	Let $t_l$ be as in Theorem \ref{thm:2}. Then we may construct a sequence of $3$-polytopes
	\[\{t_n\}_{n=l+2k, \ k\geq 1},\]
	where $t_{l+m(k-1)}\prec t_{l+mk}$ for $k\geq 1$, and each $t_n$ satisfies $\calQ_{n}$. Moreover along the sequence, for all $\epsilon>0$, one has
	\begin{equation}
	\label{eqn:op2}
	|V(t_n)|\leq\frac{n^2}{4}+\left(-1+\epsilon\right)n
	\end{equation}
	as $n\to\infty$.
\end{lemma}
Here we prove the first statement, relegating the proof of the second one \eqref{eqn:op2} to section \ref{sec:end}. The base case is just the assumption of the lemma. We take the inductive hypothesis that $t_{l+2(k-1)}$ verifies property $\calQ_{l+2(k-1)}$ of Definition \ref{def:2}. Suppose that at step $k$, $1\leq k\leq N$, we were to perform $2$-splitting on each triangle \[F_{l+2(k-1);j}, \quad 1\leq j\leq(l+2(k-1)-3)/2\]
\eqref{eqn:faces} about vertex $v_{l+2(k-1);j;3}$. That would raise by $2$ the degrees of a set of vertices of valencies $4,5,\dots, l+2(k-1)$. However, in the resulting graph, we would not be guaranteed vertices of degrees $4$ and $5$. Therefore, the $2$-splitting is taken only for
\[
2\leq j\leq (l+2(k-1)-3)/2
\]
(i.e., all these faces save $F_{l+2(k-1);1}$). We replace $F_{l+2(k-1);1}$ with the graph $S$ of Figure \ref{fig:op2}, identifying $v_{l+2(k-1);1;1}, v_{l+2(k-1);1;2}, v_{l+2(k-1);1;3}$ with $a,b,c$ respectively.

\begin{figure}[h!]
	\centering
	\includegraphics[width=3cm,clip=false]{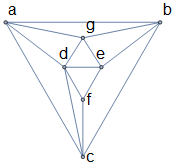}
	\caption{The graph $S$.}
	\label{fig:op2}
\end{figure}

In this way, the valencies of $v_{l+2(k-1);1;1}$ and $v_{l+2(k-1);1;2}$ increase by $2$, and they belong to the new triangle $F_{l+2k;(l+2k-3)/2}:=\{a,b,g\}$. Two vertices of valencies $4$ and $5$ are introduced, namely $d,e$. Moreover, these two belong to a triangle $F_{l+2k;1}:=\{d,e,f\}$. We also set
\[F_{l+2k;j}:=\{v_{l+2(k-1);j;1},v_{l+2(k-1);j;2},v_{l+2(k-1);j;3}''\}, \qquad 2\leq j\leq (l+2k-3)/2-1\]
where $v_{l+2(k-1);j;3}''$ is the second of the two vertices introduced in the $2$-splitting of $F_{l+2(k-1);j}$. We have thus constructed $t_{l+2k}$ satisfying $\calQ_{l+2k}$ as claimed.

In section \ref{sec:end}, it will be shown that the above produces a sequence of graphs $t_{n}$ verifying \eqref{eqn:op2}. Our goal is to optimise this method to asymptotically improve this upper bound on $|V(t_n)|$.

\subsection{The algorithm}
In section \ref{sec:pre} we have used the polyhedron $S$ as it has the $2$ triangles $\{a,b,g\}$ and $\{d,e,f\}$, with vertices of appropriate valencies $\deg(d)=5$ and $\deg(e)=\deg(a)=\deg(b)=4$. A refinement of this idea is then to pick $m$ even and use in place of $S$ a polyhedron containing $(m+2)/2$ triangles, where two vertices from each form a set of vertices of degrees
\[m+3,m+2,m+2,m+2,m+1,m,m-1,\dots,5,4.\]
We have seen in Remark \ref{rem:ap} that for $m\geq 14$, $s_{m+3}$ has the desired property $\calR_{m+3}$.

\begin{alg}\
	\label{alg:2}
	
	\textbf{Input.} A $3$-polytopal graph $t_l$ satisfying $\calQ_l$, an integer $m\geq 14$, $m\equiv 2 \pmod 4$, and a positive integer $N$.
	
	\textbf{Output.} A set of graphs $\{t_n=t_{l+mk} : 1\leq k\leq N\}$, each satisfying property $\calQ_n$, and $t_{l+m(k-1)}\prec t_{l+mk}$. These are the first $N$ entries of a sequence verifying \eqref{eqn:ordt}.
	
	\textbf{Description.} Starting from $t_l$, we perform steps $1\leq k\leq N$ as follows. The graph $t_{l+m(k-1)}$ verifies $\calQ_{l+m(k-1)}$, i.e. it has $(l+m(k-1)-3)/2$ triangular faces
	\begin{equation*}
	F_{l+m(k-1);j}=\{v_{l+m(k-1);j;1}, v_{l+m(k-1);j;2}, v_{l+m(k-1);j;3}\}, \quad 1\leq j\leq(l+m(k-1)-3)/2
	\end{equation*}
	where
	\begin{multline*}
	\{\deg(v_{l+m(k-1);j;1}) : 1\leq j\leq(l+m(k-1)-3)/2\}\ \\\cup\ \{\deg(v_{l+m(k-1);j;2}) : 1\leq j\leq(l+m(k-1)-3)/2\}=\{4,5,\dots,l+m(k-1)\}.
	\end{multline*}
	At step $k$, we $m$-split the $F_{l+m(k-1);j}$ about $v_{l+m(k-1);j;3}$, for $j=2,\dots,(l+m(k-1)-3)/2$. Next, we replace the remaining triangle $F_{l+m(k-1);1}$ with a copy of $s_{m+3}$ from Lemma \ref{le:ap}, identifying $v_{l+m(k-1);1;1}$ and $v_{l+m(k-1);1;2}$ with two adjacent vertices of degree $m+2$ in $s_{m+3}$. This is well defined, since $m+3\geq 17$, $m+3\equiv 1\pmod 4$, and $s_{m+3}$ has property $\calR_{m+3}$ (Remark \ref{rem:ap}). In this way, we have increased by $m$ the degrees of a set of vertices of valencies $4,5,\dots,l+m(k-1)$, and we have introduced $m$ new vertices of degrees $\{4,5,\dots,m+3\}$ (applying Lemma \ref{le:ap}). Moreover, these new vertices, together with $v_{l+m(k-1);1;1}$ and $v_{l+m(k-1);1;2}$, belong pairwise to $(m+2)/2$ triangles, by the construction of $s_{m+3}$ (Remark \ref{rem:ap}). We have thus obtained $t_{l+mk}$ satisfying $\calQ_{l+mk}$. Relation \eqref{eqn:subt} follows by construction.
\end{alg}

\subsection{Concluding the proofs of Theorem \ref{thm:2} and Lemma \ref{le:2}}
\label{sec:end}
It remains to show \eqref{eqn:ordt}. In Algorithm \ref{alg:2}, starting with $|V(t_l)|$ vertices, at step $k\geq 1$ we have inserted $m$ of them for each of $(l-3+m(k-1))/2-1$ $m$-splittings, plus $|V(s_{m+3})|-3$ for the operation on $F_{l+m(k-1);1}$ (i.e. replacing this triangle with a copy of $s_{m+3}$). Therefore,
\begin{align*}
|V(t_{n})|&=|V(t_l)|+\sum_{k=1}^{(n-l)/m}\left[m\cdot\frac{l-5+m(k-1)}{2}+|V(s_{m+3})|-3\right]
\\&=|V(t_l)|+\frac{m^2}{2}\sum_{k=1}^{(n-l)/m}(k-1)+\frac{m(l-5)}{2}\cdot\frac{n-l}{m}+(|V(s_{m+3})|-3)\cdot\frac{n-l}{m}
\\&\leq\frac{n^2}{4}+\left(\frac{4|V(s_{m+3})|-m^2-10m-12}{4m}+\epsilon\right)n,
\end{align*}
where as $n\to\infty$ we have bounded the constant terms via $\epsilon n$, for all $\epsilon>0$. Substituting the value \eqref{eqn:4'}, we have as $n\to\infty$
\[|V(t_{n})|\leq\frac{n^2}{4}+\left(\frac{-11m+10}{4m}+\epsilon\right)n\]
as required. The proof of Theorem \ref{thm:2} is complete.

Note that $m$ was chosen so that $m+3\geq 17$ and $m+3\equiv 1\pmod 4$ hold, in order to minimise the quantity
\[\frac{4|V(s_{m+3})|-m^2-10m-12}{4m}\]
(Lemma \ref{le:ap}) so that ultimately the coefficient of $n$ in \eqref{eqn:ordt} is as small as this method allows.

In Lemma \ref{le:2}, we had fixed instead $m=2$, and used the graph $S$ (Figure \ref{fig:op2}) in place of $s_{m+3}$. Since $|V(S)|=7$, we get \eqref{eqn:op2}. This concludes the proof of Lemma \ref{le:2}.

\paragraph{Future directions.}
The ideas behind Algorithms \ref{alg:1} and \ref{alg:2} are readily generalisable to tackle problems of a similar flavour, as shown for instance in section \ref{sec:ap}. The constructions, or a slight modification thereof, allow to minimise the total number of vertices of a graph, where certain valencies have been fixed.

\appendix
\section{Another way to present Algorithm \ref{alg:1}}
\label{appa}
The following construction of $r_n$ may be more intuitive than Algorithm \ref{alg:1}, albeit less apt for implementation. We begin by fixing $n\geq 9$ and defining a polyhedron $A(n)$ of order $n-5$ as follows. Given an initial triangle $\{v_1,v_2,v_3\}$, we add in order $v_4,v_5,\dots,v_{n-5}$ together with edges
\[v_iv_{i-1}, \quad v_iv_{i-2}, \quad v_iv_{i-3}, \qquad i=4,5,\dots,n-5\]
(splitting operations). 
The resulting $A(n)$ for $n=14, 21$ are illustrated in Figure \ref{fig:14,21}.

\begin{figure}[h!]
	\centering
	\begin{subfigure}{\numtwo}
		\centering
		\includegraphics[width=4.0cm,clip=false]{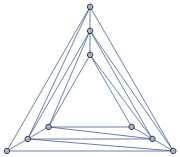}
		\caption{$A(14)$}
	\end{subfigure}
	\hspace{1.75cm}
	\begin{subfigure}{\numtwo}
		\centering
		\includegraphics[width=4.0cm,clip=false]{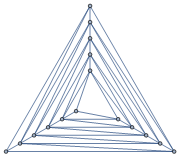}
		\caption{$A(21)$}
	\end{subfigure}
	\caption{Illustration of the construction $A(n)$.}
	\label{fig:14,21}
\end{figure}

We note that $A(9)\cong r_3$, $A(10)\cong r_{4.2}$, $A(11)\cong r_{5.2}$, and $A(12)\cong r_{6.3}$ from Figure \ref{fig:small}. For all $n\geq 11$, the degree sequence of these graphs is
\[3,4,5,6^{n-11},5,4,3\]
where the superscript is a shorthand indicating quantities of repeated numbers, e.g. $6^{n-11}$ means $n-11$ vertices of degree $6$.

Assuming $n\geq 14$, we pass from $A(n)$ to another polyhedron $B(n)$ in the following way. Firstly, we apply the splitting operation to every face of $A(n)$. This has the effect of doubling all previous vertex degrees, and introducing $2(n-5)-4$ new ones ($A(n)$ is a triangulation -- it is maximal planar) so that the sequence is now
\[6,8,10,12^{n-11},10,8,6,3^{2n-14}.\]
Secondly, we split either of the two faces containing $v_1,v_4$. This yields in particular a vertex of degree $4$. To obtain one of degree $5$, we split the two faces that are adjacent to one another and that contain $v_2,v_5$ and $v_3,v_5$ respectively. For instance, in  Figure \ref{fig:14}, inserting $v_{24}$ raises the valency of $v_{10}$ to $4$, and inserting $v_{25},v_{26}$ raises the valency of $v_{11}$ to $5$. The constructed polyhedron shall we denoted by $B(n)$. Its order is
\begin{equation}
|V(B(n))|=(n-5)+(2(n-5)-4)+3=3n-16,
\end{equation}
and its sequence
\begin{equation}
\label{eqn:seqb}
7,9,11,13,14,12,12^{n-14},10,8,6,5,4,3^{2n-13}.
\end{equation}
In \eqref{eqn:seqb}, we have purposefully isolated a subset of vertices of degree $12$
\begin{equation}
\label{eqn:v'}
V':=V'(B(n))\subset V(B(n)),
\end{equation}
with cardinality $n-14$, keeping the remaining one aside.

We have $B(14)\cong r_{14}$ (Figure \ref{fig:14}). For $n\geq 16$ our strategy is outlined as follows. The vertices in $A(n)$ have been designated to eventually correspond to ones of degree $6$ or higher in $r_n$. Following the ideas of section \ref{sec:sm}, $B(n)$ has been constructed by splitting faces of $A(n)$ that contain three of these vertices of high degree. Next, starting from $B(n)$, we split faces containing two of them. We take four vertices from $V'$ of \eqref{eqn:v'}. Via nine repeated splitting operations, we aim to raise their degrees to $15,16,17,18$. Similarly, the next four shall become of degrees $19,20,21,22$, and so forth $4k+11,4k+12,4k+13,4k+14$. 
This procedure ends when there remain either $2,3,0$, or $5$ vertices in $V'$, depending on whether $n\equiv 0,1,2$, or $3 \pmod 4$. For $n\equiv 2$ the algorithm stops here. In the other cases $n\equiv 0,1,3$, we look to apply 
further triangle splittings, to obtain vertices of degrees
\[n-1,n, \qquad n-2,n-1,n, \qquad\text{ or }\ n-4,n-3,n-2,n-1,n\]
respectively. The details have already been presented in section \ref{sec:alg1}.


\section{Illustrations for small cases of Theorem \ref{thm:1}}
\label{appb}
Here we sketch certain $3$-polytopes mentioned in sections \ref{sec:sm} and \ref{sec:alg1}.

\begin{figure}[h!]
	\begin{subfigure}{\numone}
		\centering
		\includegraphics[width=\widone,clip=false]{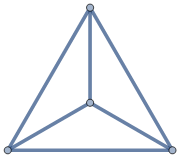}
		\caption{$r_{3}$}
	\end{subfigure}
	\begin{subfigure}{\numone}
		\centering
		\includegraphics[width=\widone,clip=false]{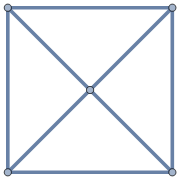}
		\caption{$r_{4.1}$}
	\end{subfigure}
	\begin{subfigure}{\numone}
		\centering
		\includegraphics[width=\widone,clip=false]{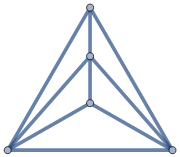}
		\caption{$r_{4.2}$}
	\end{subfigure}
	\begin{subfigure}{\numone}
		\centering
		\includegraphics[width=\widone,clip=false]{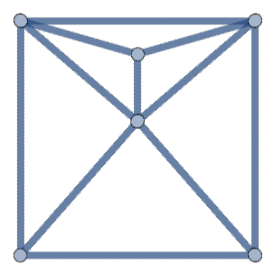}
		\caption{$r_{5.1}$}
	\end{subfigure}
	\begin{subfigure}{\numone}
		\centering
		\includegraphics[width=\widone,clip=false]{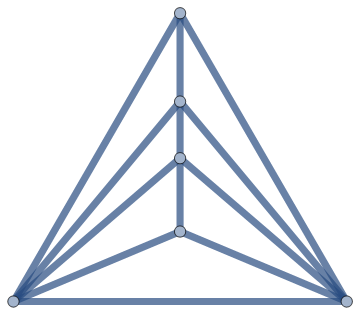}
		\caption{$r_{5.2}$}
	\end{subfigure}
\begin{subfigure}{\numone}
	\centering
	\includegraphics[width=\widone,clip=false]{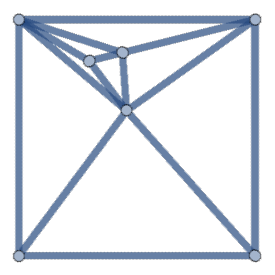}
	\caption{$r_{6.1}$}
\end{subfigure}
\begin{subfigure}{\numone}
	\centering
	\includegraphics[width=\widone,clip=false]{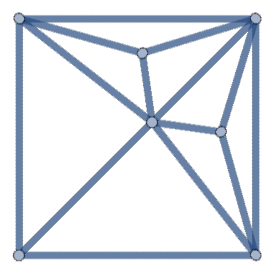}
	\caption{$r_{6.2}$}
\end{subfigure}
\begin{subfigure}{\numone}
	\centering
	\includegraphics[width=\widone,clip=false]{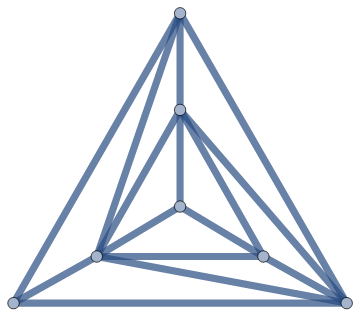}
	\caption{$r_{6.3}$}
\end{subfigure}
\begin{subfigure}{\numone}
	\centering
	\includegraphics[width=\widone,clip=false]{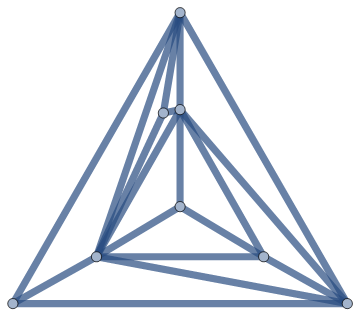}
	\caption{$r_{7.1}$}
\end{subfigure}
\begin{subfigure}{\numone}
	\centering
	\includegraphics[width=\widone,clip=false]{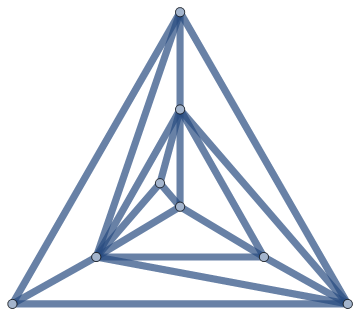}
	\caption{$r_{7.2}$}
\end{subfigure}
	\caption{The $10$ polyhedra with $p(n)=n+1$.}
	\label{fig:small}
\end{figure}

\begin{figure}[h!]
	\begin{subfigure}{\numtwo}
		\centering
		\includegraphics[width=\widtwo,clip=false]{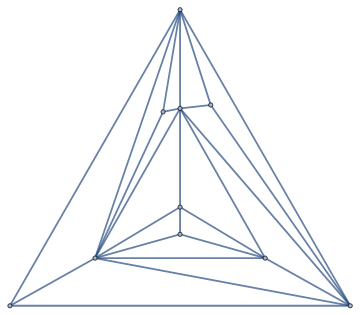}
		\caption{$r_{8}$}
	\end{subfigure}
	\begin{subfigure}{\numtwo}
		\centering
		\includegraphics[width=\widtwo,clip=false]{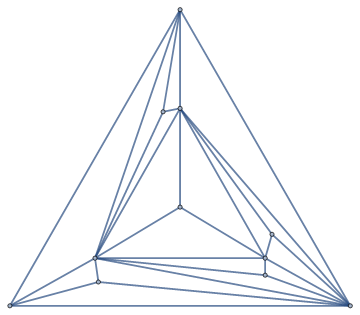}
		\caption{$r_{9}$}
	\end{subfigure}
	\begin{subfigure}{\numtwo}
		\centering
		\includegraphics[width=\widtwo,clip=false]{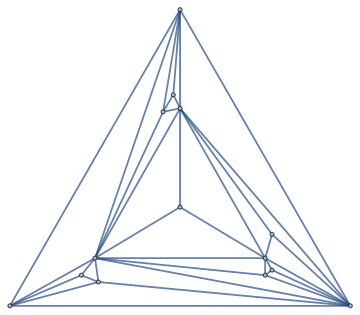}
		\caption{$r_{10}$}
	\end{subfigure}
	\begin{subfigure}{\numtwo}
		\centering
		\includegraphics[width=\widtwo,clip=false]{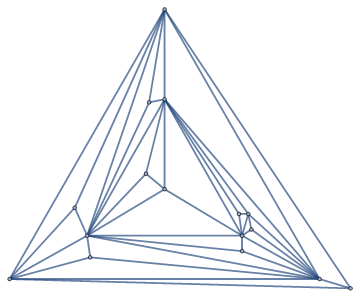}
		\caption{$r_{11}$}
	\end{subfigure}
	\begin{subfigure}{\numtwo}
		\centering
		\includegraphics[width=\widtwo,clip=false]{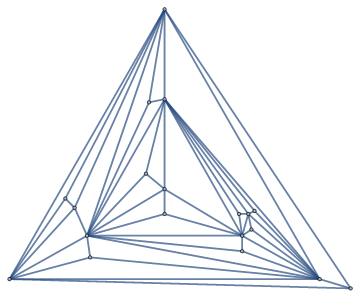}
		\caption{$r_{12}$}
	\end{subfigure}
	\begin{subfigure}{\numtwo}
		\centering
		\includegraphics[width=\widtwo,clip=false]{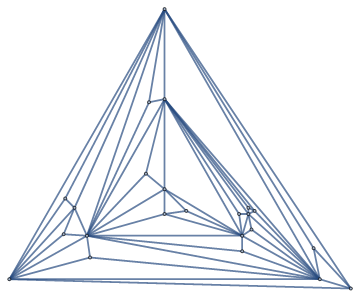}
		\caption{$r_{13}$}
	\end{subfigure}
	\caption{Examples of $3$-polytopes $r_n$ satisfying $\calP_n$, for $8\leq n\leq 13$.}
\label{fig:8-13}
\end{figure}
\begin{figure}[h!]
	\centering
	\includegraphics[width=7cm,clip=false]{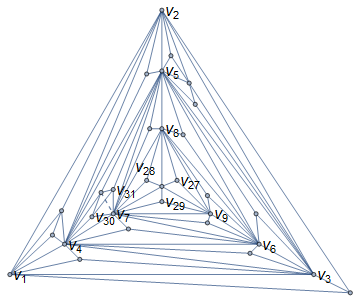}
	\caption{A graph $r_{15}$ satisfying $\calP_{15}$, obtained from $r_{14}$ (Figure \ref{fig:14}) by deleting the dashed edge and inserting $v_j$, $27\leq j\leq 31$, and their incident edges.}
	\label{fig:15}
\end{figure}

\clearpage
\bibliographystyle{plain}
\bibliography{bibgra}

\end{document}